\documentclass[12pt, reqno, twoside, letterpaper]{amsart}

\usepackage{paperstyle}

\usepackage{graphicx}
\usepackage{float}
\usepackage{mathtools}
\usepackage{bbm}
\usepackage{booktabs}

\newcommand{\Dfun}{\boldsymbol D} 
\newcommand{\Efun}{\boldsymbol E} 
\newcommand{\Om}{\boldsymbol\Omega} 
\newcommand{\om}{\boldsymbol\omega} 
\newcommand{\omflat}{\om_{\flat}} 
\newcommand{\omsharp}{\om_{\sharp}} 
\newcommand{\rzero}{\boldsymbol\rho_0} 
\newcommand{\azero}{{\boldsymbol a}_0} 

\newcommand{\szero}{\boldsymbol\sigma_0} 
\newcommand{\tszero}{\boldsymbol{\widetilde\sigma}_{\boldsymbol0}} 
\newcommand{\Thstar}{\boldsymbol\Theta} 
\newcommand{\Zfun}{\boldsymbol Z} 
\newcommand{\zalt}{\zeta_{\mathrm{alt}}} 

\DeclareMathOperator{\Li}{Li}

\title{Zeros of the Dirichlet series of even zeta values}
\author{William D. Banks}
\address{Department of Mathematics\\
University of Missouri\\
Columbia MO 65211\\
USA}
\email{bankswd@missouri.edu}

\keywords{Riemann zeta function, Dirichlet series, functional equation, zeros of Dirichlet series, $a$-points}
\subjclass[2020]{11M06, 11M26, 11M41}

\date{\today}

\begin{document}

\begin{abstract}
We give a complete unconditional description of the zero set of the
Dirichlet series
$\Dfun(s)\defeq\sum_{n\ge1}\zeta(2n)\,n^{-s}$,
which continues meromorphically to $\C$ with a single simple pole at
$s=1$. The series possesses neither an Euler product nor a self-dual
functional equation, and descriptions with this level of completeness are
exceedingly rare for such series. The key input is an exact
functional equation of Hecke type, obtained from the Lipschitz
summation formula, which expresses $\Dfun$ in the left half-plane as
a gamma factor times a dual series over the complex logarithms of
the perfect squares; Riemann's functional equation appears as a
single column of the dual series. The zeros fall into four families.
The half-plane $\sigma\ge\szero=1.5001\cdots$ is zero-free, and
$\Dfun$ has a unique real zero $\rzero=0.2004\cdots$. The zeros in
the critical strip are perturbed $a$-points
of $\zeta$ for values of $a$ near $-(\zeta(2)-1)$, and their
counting function obeys a Riemann--von Mangoldt law. The remaining
zeros form two complex-conjugate strings that recede into the left
half-plane along explicit rays, are eventually simple, and satisfy
an asymptotic with geometrically decaying error. The string geometry
is governed by interference between the two smallest frequencies of
the dual series, $2\log2$ contributed by the entire part of $\Dfun$
and $\pm2\pi i$ contributed by $\zeta$. No hypothesis of Riemann
type is assumed at any point.
\end{abstract}

\maketitle

\tableofcontents

\newpage{\Large\section{Introduction}}

A Dirichlet series that possesses neither an Euler product nor a
self-dual functional equation escapes much of the classical theory.
Davenport and Heilbronn~\cite{DH} showed that such a series can
vanish even in its half-plane of absolute convergence, yet a complete
unconditional description of an entire zero set is seldom possible.
This paper gives one such description. Following Riemann,
$s=\sigma+it$ denotes a complex variable, with
$\sigma \defeq \Re\,s$ and $t \defeq \Im\,s$. We investigate the
Dirichlet series
\be\label{eq:def}
\Dfun(s)\defeq\sum_{n\in\N}\frac{\zeta(2n)}{n^{s}}\qquad(\sigma>1),
\ee
where $\zeta$ is the Riemann zeta function. As $\zeta(2n)\to1$
and the coefficients are positive, the series in \eqref{eq:def} has
abscissa of convergence equal to one. On the other hand, since
$\zeta(2n)-1\ll 4^{-n}$, the related series
\[
\Efun(s)\defeq\sum_{n\in\N}\frac{\zeta(2n)-1}{n^{s}}
\]
converges absolutely for all $s\in\C$, hence \emph{$\Efun$ is entire}.
By the decomposition
\be\label{eq:split}
\Dfun(s)=\zeta(s)+\Efun(s),
\ee
we infer that $\Dfun$ continues meromorphically to $\C$,
the only singularity being a simple pole at $s=1$ of residue one.

Writing $\zeta(2n)-1=\sum_{k\ge 2}k^{-2n}$ and interchanging sums,
we have
\be\label{eq:polylog}
\Dfun(s)=\sum_{k\in\N}\Li_{s}(k^{-2})
\mand
\Efun(s)=\sum_{k\ge 2}\Li_{s}(k^{-2}),
\ee
where
\[
\Li_{s}(u)\defeq\sum_{m\in\N}u^{m}m^{-s}
\]
is the polylogarithm of order $s$.
For $\Efun$, the underlying double sum is absolutely convergent at
every $s\in\C$, so the interchange is justified and
\eqref{eq:polylog} holds for $\Efun$ throughout $\C$.
Since $\Li_{s}(1)=\zeta(s)$ and $\Efun$ is entire, the identity for
$\Dfun$ in \eqref{eq:polylog} follows from \eqref{eq:split} when $\sigma>1$,
and it extends to all of $\C$ by analytic continuation.

Special values of $\Dfun$ and of $\Efun$ at integer arguments belong
to the classical family of rational zeta series (see \cite{SC,BBC}).
For example, we have
\[
\Efun(0)=\sum_{n\in\N}(\zeta(2n)-1)=\tfrac34
\mand
\Efun(1)=\sum_{n\in\N}\frac{\zeta(2n)-1}{n}=\log 2.
\]
Combined with \eqref{eq:split}, the first evaluation implies
$\Dfun(0)=\zeta(0)+\tfrac34=\tfrac14$, and the second (as $\Efun$ is
entire) yields the Laurent expansion of $\Dfun$ near $s=1$, namely,
\[
\Dfun(s)=(s-1)^{-1}+\gamma+\log 2+O(s-1).
\]
Identities of this kind capture the behavior of $\Dfun$ at
individual points, but they say nothing about where $\Dfun$
vanishes.

Our principal aim is to determine the global structure of the
zero set of~$\Dfun$. The key input is an exact functional equation
of Hecke type (Theorem~\ref{thm:fe}), obtained by applying the
Lipschitz summation formula to each summand in \eqref{eq:polylog}.
The dual object is a double series over the \emph{frequency set}
\be\label{eq:freq}
\Om\defeq
\bigl\{\om_{k,\ell}=\log k^2+2\pi i \ell:k\in\N,~\ell\in\Z\bigr\}
\setminus\{0\},
\ee
that is, over all complex logarithms of the perfect squares. The
$k=1$ column of the dual series recovers Riemann's functional equation.

Interference between the smallest real frequency
\be\label{eq:omflat}
\omflat\defeq\om_{2,0}=2\log2\qquad\text{(contributed by $\Efun$)}
\ee
and the smallest complex frequencies
\be\label{eq:omsharp}
\pm\omsharp\defeq\om_{1,\pm1}=\pm2\pi i\qquad\text{(contributed by $\zeta$)}
\ee
determines every zero of sufficiently large modulus in the left half-plane.
The resulting picture (Figure~\ref{fig:zeros})
displays several notable features: a zero-free right half-plane
(Proposition~\ref{prop:right}), zeros in
the critical strip that shadow the $a$-points of~$\zeta$ for
$a\approx-(\zeta(2)-1)$, with a count obeying a Riemann--von Mangoldt
law (Theorem~\ref{thm:rvm}), an intermediate zero-free zone
(Proposition~\ref{prop:middle}), and two rectilinear strings of
left-plane zeros with explicit asymptotics
(Theorem~\ref{thm:strings}). On the real axis, $\Dfun$ vanishes at
\[
\rzero=0.200411339\cdots
\]
(Proposition~\ref{prop:real}), and this is its only real zero
(Proposition~\ref{prop:mono}).

\begin{figure}[ht]
\centering
\includegraphics[width=4.5truein]{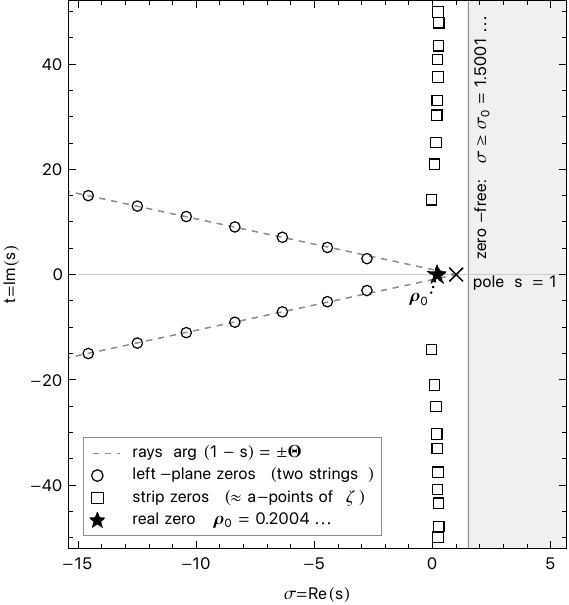}
\caption{Computed zeros of $\Dfun$ in the window $|t|\le50$. The two
strings of left-plane zeros recede along the dashed rays
$\arg(1-s)=\pm\Thstar$. The strip zeros shadow $a$-points
of~$\zeta$. The star marks the real zero $\rzero=0.2004\cdots$, and
the cross marks the pole at~$s=1$. The shaded half-plane
$\sigma\ge\szero$ is zero-free.
Tables~\ref{tab:left} and~\ref{tab:strip} extend beyond the plotted
range.}
\label{fig:zeros}
\end{figure}

The function $\Dfun$ is among the simplest Dirichlet series built
from special values of $\zeta$, and it serves as a test case for the
broader problem raised at the outset, that of describing the global
geometry of the zeros of a Dirichlet series in the absence of an
Euler product and a self-dual functional equation. Beyond $\Dfun$
itself, the results make two further points of contact with the
literature. The strip zeros realize the value-distribution theory of
$a$-points of $\zeta$ with the constant $a$ replaced by a bounded
analytic function, and the left-plane strings exhibit explicitly the
zero geometry of two-term exponential sums.

{\Large\section{Summary of main results}\label{sec:summary}}

This section collects the main results in one place. Every object
that appears below is defined at the indicated later point in the
paper, and nothing stated here is used in the proofs, so the section
can be read independently of the rest.

\bigskip\noindent{\bf Continuation and symmetry.}
The decomposition $\Dfun=\zeta+\Efun$ of \eqref{eq:split}, in which
$\Efun$ is entire, continues $\Dfun$ meromorphically to $\C$; the
only singularity is a simple pole at $s=1$ of residue one. The
coefficients $\zeta(2n)$ are real and positive, so
$\overline{\Dfun(\overline s)}=\Dfun(s)$, and the zero set of
$\Dfun$ is symmetric about the real axis.

\bigskip\noindent{\bf The functional equation.}
All results derive from an exact functional equation of Hecke type
(Theorem~\ref{thm:fe}). In the compact form \eqref{eq:hecke}, it
reads
\[
\Dfun(s)=\Gamma(1-s)\,\Zfun(1-s)
\qquad(\sigma<0),
\]
where the dual series $\Zfun$ runs over the frequency set $\Om$ of
\eqref{eq:freq}, consisting of the complex logarithms of the perfect
squares, under the mandatory grouping of terms explained in
Remark~\ref{rem:order}. The $k=1$ column of the dual series is
Riemann's functional equation \eqref{eq:rfe}. The frequencies of
smallest modulus are $\omflat=2\log2$, contributed by $\Efun$, and
$\pm\omsharp=\pm2\pi i$, contributed by $\zeta$, as recorded in
\eqref{eq:omflat} and \eqref{eq:omsharp}; interference between the
frequencies determines the left-plane zeros.

\bigskip\noindent{\bf The zero set.}
Figure~\ref{fig:zeros} displays the computed zeros, and
Table~\ref{tab:summary} summarizes the description of the zero set
by region.

\bigskip

\begin{table}[ht]
\centering
\setlength{\tabcolsep}{12pt}
\begin{tabular}{|l|l|l|}
\hline
\rule{0pt}{2.6ex}%
region & zeros of $\Dfun$ & result \\[4pt]
\hline
\rule{0pt}{2.6ex}%
$\sigma\ge\szero$ & none & Proposition~\ref{prop:right} \\
real axis & $\rzero=0.2004\cdots$
& Propositions~\ref{prop:real}, \ref{prop:mono} \\
$-\tfrac14\le\sigma\le\szero$ & perturbed $a$-points of $\zeta$
& Theorem~\ref{thm:rvm} \\
$-C\le\sigma\le-\delta$ & finitely many & Proposition~\ref{prop:middle} \\
$\sigma\le-\tszero$ & two simple strings & Theorem~\ref{thm:strings} \\[5pt]
\hline
\end{tabular}
\vspace{8pt}
\caption{The zero set of $\Dfun$ by region. Thresholds and constants
are defined at the indicated results.}
\label{tab:summary}
\end{table}

In the right half-plane, we have $\Dfun(s)\ne0$ for
$\sigma\ge\szero$, where $\szero=1.5001\cdots$ is the abscissa
characterized in Proposition~\ref{prop:right}. On the real axis,
$\Dfun(\sigma)>0$ holds for $\sigma\le0$ and for $\sigma>1$, while
$\Dfun(1^{-})=-\infty$, so $\Dfun$ vanishes at some
$\rzero\in(0,1)$. Moreover, $\Dfun$ decreases strictly on $(0,1)$,
so $\rzero$ is the unique real zero
(Propositions~\ref{prop:real}, \ref{prop:mono}); numerically, we have
$\rzero=0.2004\cdots$.

In the critical strip, every zero $\rho$ with $\Re\,\rho\ge-\delta$
satisfies $\zeta(\rho)=-\Efun(\rho)$ and is therefore an $a$-point
of $\zeta$ for a value of $a$ within the explicit distance
$r(\delta)$ of $\azero=-(\zeta(2)-1)=-0.6449\cdots$; see
\eqref{eq:Econst}. The count of these zeros obeys a Riemann--von
Mangoldt law (Theorem~\ref{thm:rvm}), namely,
\[
N_{\Dfun}(T)=\frac{T}{2\pi}\log\frac{T}{2\pi}-\frac{T}{2\pi}
+O(\log T),
\]
in agreement to this order with the counting function of the zeros
of $\zeta$ itself. In addition, every fixed strip
$-C\le\sigma\le-\delta$ contains only finitely many zeros
(Proposition~\ref{prop:middle}).
In the left half-plane, there exists $\tszero>0$ with the following
property. Every zero of $\Dfun$ with $\sigma\le-\tszero$ is simple,
and these zeros form two complex-conjugate strings
\[
{\boldsymbol\rho}_{m}=1-\frac{i\pi(2m+1)}{\sL}+O(\er^{-\kappa m})
\qquad(m\ge m_{0}),
\]
where $\sL=\log(\omsharp/\omflat)$ and $\kappa>0$
(Theorem~\ref{thm:strings}, Section~\ref{sec:left}). The strings
recede along the rays $\arg(1-s)=\pm\Thstar$, consecutive zeros on
each string have limiting spacing $2\pi/|\sL|$, and the number of
left-plane zeros of modulus at most $R$ equals $(|\sL|/\pi)R+O(1)$.

\bigskip\noindent{\bf Constants.}
The constants of Theorem~\ref{thm:strings} evaluate as
\[
\sL=\log\frac{\pi}{\log2}+\frac{i\pi}{2},
\qquad
\Thstar=\frac{\pi}{2}-\arg\sL=0.7660\cdots,
\qquad
\frac{2\pi}{|\sL|}=2.8825\cdots,
\]
and the exponential rate derives from the tail exponent
$\eta=0.0237661\cdots$ of Lemma~\ref{lem:tail} and
Remark~\ref{rem:eta}, any
$\kappa\le\eta\pi^{2}/(2|\sL|^{2})=0.0246\cdots$ being admissible.

\bigskip\noindent{\bf Numerics.}
Table~\ref{tab:left} compares the first zeros of the lower string
with the asymptotic of Theorem~\ref{thm:strings}, the error decaying
geometrically, and Table~\ref{tab:strip} pairs each strip zero of
ordinate below $75$ with the $a$-point of $\zeta$ that it shadows.
Section~\ref{sec:comp} describes the computations.

\bigskip\noindent{\bf Organization.}
Section~\ref{sec:fe} proves the functional equation.
Section~\ref{sec:zerofree} treats the zero-free half-plane, the real
zeros, and the intermediate strips. Section~\ref{sec:left}
establishes the string theorem, and Section~\ref{sec:strip} the
counting law. Section~\ref{sec:comp} records the computational
details, and Section~\ref{sec:concl} poses three questions.

{\Large\section{The functional equation}\label{sec:fe}}

This section establishes the functional equation of Hecke type for
$\Dfun$. Theorem~\ref{thm:fe} expresses $\Dfun(s)$ in the half-plane
$\sigma<0$ as a gamma factor times a dual series over the frequency
set $\Om$ of \eqref{eq:freq}, Remark~\ref{rem:order} explains why the
dual series converges only under a mandatory grouping of terms, and
Remark~\ref{rem:spectral} locates the frequencies on the Mellin transform side.

Our starting point is the Lipschitz summation formula
\cite{Lipschitz}, in the form of Hurwitz's formula for the periodic
zeta function given in \cite[\S1.11]{HTF} and
\cite[\S25.11(viii), \S25.12(ii)]{DLMF}; an elementary proof appears
in \cite{KR}. For every $\lambda>0$, we have
\be\label{eq:lipschitz}
\Li_{s}(\er^{-\lambda})
=\Gamma(1-s)\sum_{\ell\in\Z}(\lambda+2\pi i\,\ell)^{\,s-1}
\qquad(\sigma<0),
\ee
where powers are computed using the principal branch of the logarithm.
The series in \eqref{eq:lipschitz} converges absolutely for
$\sigma<0$,
since $|\lambda+2\pi i\ell|^{\sigma-1}\ll|\ell|^{\sigma-1}$ for each
$\ell\ne0$. 

The degenerate case $\lambda=0$ of \eqref{eq:lipschitz}, with the
term $\ell=0$ removed, is the asymmetric version of Riemann's
functional equation,
\be\label{eq:rfe}
\zeta(s)=\Gamma(1-s)\sum_{\ell\ne0}(2\pi i \ell)^{\,s-1}
=\cX(s)\,\zeta(1-s),
\ee
where
\[
\cX(s)\defeq2^{s}\pi^{s-1}\sin\tfrac{\pi s}{2}\,\Gamma(1-s).
\]

\bigskip

\begin{theorem}\label{thm:fe}
For $\sigma<0$, we have
\be\label{eq:fe}
\Dfun(s)=\Gamma(1-s)\sum_{k=1}^{\infty}\ \sideset{}{'}\sum_{\ell\in\Z}\,
(2\log k+2\pi i \ell)^{s-1},
\ee
where the prime indicates that the term $(k,\ell)=(1,0)$ is omitted.
The inner sum converges absolutely for each $k$, as does the outer
sum over $k$. Equivalently,
\be\label{eq:fe2}
\Dfun(s)=\cX(s)\,\zeta(1-s)+\Gamma(1-s)
\sum_{k=2}^{\infty}\sum_{\ell\in\Z}(2\log k+2\pi i\ell)^{s-1}.
\ee
\end{theorem}

\begin{proof}
Fix $s$ with $\sigma<0$. Apply \eqref{eq:lipschitz} with
$\lambda=2\log k$ to each term of \eqref{eq:polylog} with $k\ge2$,
and \eqref{eq:rfe} to the term $k=1$. It remains to justify summing
the identities over $k$. The left-hand sides are absolutely summable,
since for $0<z\le\tfrac14$ we have
\[
\bigl|\Li_{s}(z)\bigr|\le\sum_{m\in\N}z^{m}m^{|\sigma|}
\le z\sum_{m\in\N}4^{-(m-1)}m^{|\sigma|}=z\cdot C(\sigma)\quad\text{(say)},
\]
and thus
\[
\sum_{k\ge2}|\Li_{s}(k^{-2})|\le
C(\sigma)\sum_{k\ge2}k^{-2}<\infty.
\]
Since $\Gamma(1-s)\ne0$, the
inner sums in \eqref{eq:fe} are likewise absolutely summable
over $k$, and summing the identities termwise yields \eqref{eq:fe}.
\end{proof}

\begin{remark}\label{rem:order}
We emphasize that the double series in \eqref{eq:fe} is
\emph{not} absolutely convergent as a sum over the frequency set
$\Om$ of \eqref{eq:freq}.
For fixed $\ell$, the sum
\[
\sum_{k\ge2}\bigl|2\log k+2\pi i\ell\bigr|^{\sigma-1}
\]
diverges for every $\sigma$, since $(\log k)^{-u}$ is not summable
over $k\ge2$ for any $u\in\R$. Convergence in \eqref{eq:fe} derives from the
cancellation that occurs in each inner sum over $\ell$. Indeed, for
$k\ge 2$,
\[
\sum_{\ell\in\Z}(2\log k+2\pi i \ell)^{s-1}
=\frac{\Li_{s}(k^{-2})}{\Gamma(1-s)}\ll k^{-2}.
\]
Thus \eqref{eq:fe} should be interpreted as a Hecke-type
identity
\be\label{eq:hecke}
\Dfun(s)=\Gamma(1-s)\,\Zfun(1-s)
\qquad\text{with}\quad
\Zfun(w)\defeq\sum_{\om\in\Om}\om^{-w},
\ee
in which the dual series $\Zfun$ carries a mandatory grouping of terms.
The set $\Om$ collects every value of $\log k^{2}$ over $k\in\N$,
one for each branch of the logarithm, so \eqref{eq:hecke} trades the
real frequencies $\log n$ of the Dirichlet series \eqref{eq:def} for
the complex logarithms of the perfect squares.
\end{remark}

\begin{remark}\label{rem:spectral}
The frequencies in \eqref{eq:freq} can be seen on the Mellin
transform side. For $u>0$, summing a geometric series in each term yields
\[
\sum_{n\in\N}\zeta(2n)\er^{-nu}
=\sum_{m\in\N}\frac{\er^{-u}}{m^{2}-\er^{-u}},
\]
which is the classical expansion
$\sum_{n\in\N}\zeta(2n)x^{2n}=\tfrac12(1-\pi x\cot\pi x)$ evaluated at
$x=\er^{-u/2}$. Taking Mellin transforms, we obtain, for $\sigma>1$,
\[
\Gamma(s)\,\Dfun(s)
=\int_{0}^{\infty}u^{s-1}\Bigl(\frac{1}{\er^{u}-1}+G(u)\Bigr)\dd u
\qquad\text{with}\quad
G(w)\defeq\sum_{m\ge2}\frac{\er^{-w}}{m^{2}-\er^{-w}}.
\]
Here $G$ is meromorphic on \C with simple poles, each of residue one,
exactly at the points $w=-\om_{k,\ell}$ with $k\ge2$, while
$1/(\er^{w}-1)$ contributes the poles $w=-\om_{1,\ell}$ for $\ell\ne0$,
together with a pole at $w=0$ reflecting the pole of $\Dfun(s)$ at $s=1$.
Theorem~\ref{thm:fe} is the statement that the analytic continuation
of the Mellin transform is assembled from these poles, in perfect
analogy with Riemann's first proof of \eqref{eq:rfe}.
\end{remark}

{\Large\section{Zero-free regions and real zeros}\label{sec:zerofree}}

This section treats the zeros of $\Dfun$ away from the critical
strip, together with the real zero. Proposition~\ref{prop:right}
produces a zero-free half-plane $\sigma\ge\szero$,
Proposition~\ref{prop:real} locates a real zero $\rzero\in(0,1)$,
Proposition~\ref{prop:mono} shows that it is the only one, and
Proposition~\ref{prop:middle} shows that each strip
$-C\le\sigma\le-\delta$ contains only finitely many zeros.

\bigskip

\begin{proposition}\label{prop:right}
Let $\szero=1.500127440\cdots$ be the unique real root of
\[
\sum_{n\ge2}\zeta(2n)n^{-\sigma}=\zeta(2).
\]
Then $\Dfun(s)\ne0$ for $\sigma\ge\szero$.
\end{proposition}

\begin{proof}
Splitting off the term $n=1$ in \eqref{eq:def}, we have
\be\label{eq:righttail}
\Dfun(s)-\zeta(2)=\sum_{n\ge2}\frac{\zeta(2n)}{n^{s}}\qquad(\sigma>1).
\ee
By the triangle inequality and the positivity of the coefficients
$\zeta(2n)$,
\be\label{eq:righttri}
\bigl|\Dfun(s)-\zeta(2)\bigr|\le h(\sigma)
\qquad\text{with}\quad
h(\sigma)\defeq\sum_{n\ge2}\frac{\zeta(2n)}{n^{\sigma}}.
\ee
The function $h$ is strictly decreasing on $(1,\infty)$, with
$h(1^{+})=\infty$ and $h(\infty)=0$; hence $\szero$ is the
unique root of $h(\sigma)=\zeta(2)$.

For all $\sigma>\szero$, the bound \eqref{eq:righttri} shows
\[
|\Dfun(s)-\zeta(2)|\le h(\sigma)<h(\szero)=\zeta(2)
\]
and thus $\Dfun(s)\ne 0$. Suppose now that
$\Dfun(s)=0$ for some $s$ with $\sigma=\szero$. Then equality holds
in~\eqref{eq:righttri}, leading to
\[
\biggl|\sum_{n\ge2}\frac{\zeta(2n)}{n^{s}}\biggr|
=\sum_{n\ge2}\frac{\zeta(2n)}{n^{\sigma}},
\]
which forces the phases $n^{-it}$ to have the same value for all
$n\ge 2$. This occurs only if $t=0$, that is, only if $s=\szero$.
However, \eqref{eq:righttail} and the definition of $\szero$ yield
\[
\Dfun(\szero)=\zeta(2)+h(\szero)=2\,\zeta(2)>0,
\]
a contradiction. Therefore $\Dfun(s)\ne0$ for all $\sigma\ge\szero$.
\end{proof}

\begin{proposition}\label{prop:real}
We have $\Dfun(\sigma)>0$ for all $\sigma\le 0$ and all
$\sigma>1$, and $\Dfun(1^{-})=-\infty$. Consequently,
$\Dfun(\rzero)=0$ for some $\rzero\in(0,1)$; numerically,
$\rzero=0.200411339\cdots$.
\end{proposition}

\begin{proof}
For $\sigma>1$, every term of \eqref{eq:def} is positive, whence
$\Dfun(\sigma)>0$. Because $\zeta(1^{-})=-\infty$ and $\Efun$ is
bounded near $s=1$, we have $\Dfun(1^{-})=-\infty$, and since
$\Dfun(0)=\tfrac14$, a zero $\rzero\in(0,1)$ exists by
continuity.

\begin{figure}[ht]
\centering
\includegraphics[width=4.5truein]{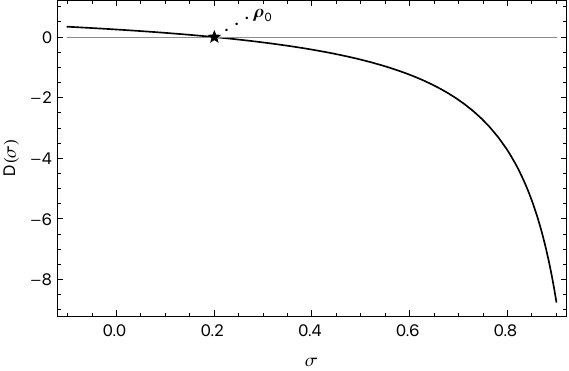}
\caption{$\Dfun(\sigma)$ on the real segment $(-0.1,\,0.9)$. The sign
change occurs at $\rzero=0.2004\cdots$, the only real zero of
$\Dfun$.}
\label{fig:real}
\end{figure}

It remains to prove that $\Dfun(-x)>0$ for all $x\ge0$. Since
$\zeta(2n)-1\ge4^{-n}$ for every $n\in\N$, we have
\[
\Efun(-x)\ge\sum_{m\in\N}m^{x}4^{-m}=\Li_{-x}\bigl(\tfrac14\bigr).
\]
Suppose first that $x\ge1.3$. The map $u\mapsto u^{x}4^{-u}$
increases on $(0,x/\log4]$ and decreases thereafter, attaining the
maximum value $M(x)\defeq(x/(\er\log4))^{x}$. Comparing the sum with
the integral, we find that
\[
\Li_{-x}\bigl(\tfrac14\bigr)\ge\int_{0}^{\infty}u^{x}4^{-u}\dd u-M(x)
=\frac{\Gamma(x+1)}{(\log4)^{x+1}}-M(x)
\ge\frac{\Gamma(x+1)}{2(\log4)^{x+1}},
\]
where the last step follows from Stirling's bound
$\Gamma(x+1)\ge\sqrt{2\pi x}\,(x/\er)^{x}$ together with the
inequality $\sqrt{2\pi x}\ge2\log4$, valid for $x\ge1.3$.
On the other hand, \eqref{eq:rfe} yields the upper bound
\[
|\zeta(-x)|\le\pi^{-1}(2\pi)^{-x}\Gamma(x+1)\,\zeta(x+1).
\]
The factors $\Gamma(x+1)$ cancel in the ratio, and we obtain
\[
\frac{|\zeta(-x)|}{\Efun(-x)}
\le 4\,\zeta(x+1)\Bigl(\frac{\log4}{2\pi}\Bigr)^{x+1}
\le 4\,\zeta(2.3)\Bigl(\frac{\log4}{2\pi}\Bigr)^{2.3}=0.177\cdots<1,
\]
since both factors on the right decrease for $x\ge 1.3$;
therefore,
\[
\Dfun(-x)\ge\Efun(-x)-|\zeta(-x)|>0.
\]

Now suppose that $0\le x\le1.3$. The coefficients of $\Efun$ are
positive, so $\Efun$ decreases on the real axis, and
$\Efun(-x)\ge\Efun(0)=\tfrac34$. Since $|\zeta(-x)|\le\tfrac12$ on
this interval, we conclude that $\Dfun(-x)\ge\tfrac14>0$.
\end{proof}

Proposition~\ref{prop:real} locates $\rzero$ but does not by itself
exclude further real zeros in $(0,1)$. The following result supplies
that, and completes the description of the zero set on the real axis
shown in Figure~\ref{fig:real}.

\begin{proposition}\label{prop:mono}
The function $\Dfun$ is strictly decreasing on $(0,1)$. Consequently,
$\rzero$ is the unique real zero of $\Dfun$.
\end{proposition}

\begin{proof}
Recall that Dirichlet's alternating zeta function
\[
\zalt(s)\defeq\sum_{n\in\N}\frac{(-1)^{n-1}}{n^{s}}=(1-2^{1-s})\zeta(s)
\qquad(\sigma>0)
\]
is positive and nondecreasing for real $s=\sigma\in(0,\infty)$; see \cite{AL}
(and \cite[Theorem~3]{vdL} for the original statement). We have
\[
-\zeta(\sigma)=\frac{\zalt(\sigma)}{2^{1-\sigma}-1}\qquad(0<\sigma<1),
\]
the denominator being positive. Let $0<\sigma<\sigma'<1$. Then
$\zalt(\sigma')\ge\zalt(\sigma)>0$,
and $0<2^{1-\sigma'}-1<2^{1-\sigma}-1$, therefore
\[
-\zeta(\sigma')=\frac{\zalt(\sigma')}{2^{1-\sigma'}-1}
\ge\frac{\zalt(\sigma)}{2^{1-\sigma'}-1}
>\frac{\zalt(\sigma)}{2^{1-\sigma}-1}=-\zeta(\sigma).
\]
Thus $\zeta$ is strictly decreasing on $(0,1)$. As noted in the proof
of Proposition~\ref{prop:real}, the coefficients of $\Efun$ are
positive, so $\Efun$ is decreasing on the real axis; by
\eqref{eq:split}, $\Dfun$ is strictly decreasing on $(0,1)$.

Proposition~\ref{prop:real} gives $\Dfun(\rzero)=0$ for some
$\rzero\in(0,1)$, and strict monotonicity makes that zero unique in
the interval. The same proposition gives $\Dfun(\sigma)>0$ for
$\sigma\le0$ and for $\sigma>1$, and $s=1$ is a pole; hence $\rzero$
is the only real zero of $\Dfun$.
\end{proof}

\begin{remark}
An earlier version of this paper stated Proposition~\ref{prop:mono}
as a conjecture. The argument above is due to Qiping Zhou, who
observed that the statement follows from the monotonicity of
Dirichlet's alternating zeta function on the positive axis.
\end{remark}

We turn from the real axis to the strips of the left half-plane,
where $\zeta$ outgrows the bounded function $\Efun$ along every
vertical line.

\begin{proposition}\label{prop:middle}
For all $C>\delta>0$, the strip $-C\le\sigma\le-\delta$ contains
only finitely many zeros of $\Dfun$.
\end{proposition}

\begin{proof}
The function $\Efun$ satisfies $|\Efun(s)|\le\Efun(-C)$ on the
half-plane $\sigma\ge-C$. Using \eqref{eq:rfe}, together with
the asymptotic $|\cX(\sigma+it)|\asymp(|t|/2\pi)^{1/2-\sigma}$
and the elementary lower bound $|\zeta(w)|\ge\zeta(2u)/\zeta(u)$
for $u\defeq\Re\,w>1$, we have
\be\label{eq:zetabound}
|\zeta(s)|=|\cX(s)|\,|\zeta(1-s)|\ggsym{C}|t|^{1/2+\delta}\,
\frac{\zeta(2+2\delta)}{\zeta(1+\delta)}
\qquad(-C\le\sigma\le-\delta,~|t|\ge2).
\ee
Hence $|\zeta(s)|>|\Efun(s)|$, and so $\Dfun(s)\ne0$, once $|t|$ is
large enough. The zeros of $\Dfun$ in the strip therefore lie in a
compact rectangle, and since $\Dfun\not\equiv0$, they are finite in
number.
\end{proof}

{\Large\section{Zeros in the left half-plane}\label{sec:left}}

This section treats the zeros of $\Dfun$ in the left half-plane.
Theorem~\ref{thm:strings} shows that all zeros of sufficiently large
modulus recede along two explicit rays emanating from the point
$s=1$, locates each zero to within an exponentially small error, and
determines the limiting spacing between consecutive zeros. The
mechanism is interference. For $\sigma<0$, the dual series of
Theorem~\ref{thm:fe} is controlled by its two leading frequencies,
one contributed by the entire part $\Efun$ and the other by $\zeta$,
and the zeros settle on the interface where the two contributions
balance. Throughout, we write
\[
w\defeq1-s,\qquad u\defeq\Re\,w,\qquad v\defeq\Im\,w,
\]
so that $\sigma<0$ corresponds to $u>1$, and we recall from
\eqref{eq:omflat} and \eqref{eq:omsharp} the two leading frequencies
of $\Om$,
\[
\omflat=\om_{2,0}=2\log2,\qquad\omsharp=\om_{1,1}=2\pi i.
\]
The formulation of the theorem involves the constants
\[
\sL\defeq\log(\omsharp/\omflat)
\mand
\Thstar\defeq\frac{\pi}{2}-\arg\sL,
\]
where the principal branch of the logarithm is used. We set
\[
A\defeq\Re(\sL)=\log\frac{\pi}{\log2}
\mand
L\defeq|\sL|,
\]
so that
\be\label{eq:omom}
\sL=A+i\pi/2=L\,\er^{i(\pi/2-\Thstar)}.
\ee

\bigskip

\begin{theorem}\label{thm:strings}
There exists $\tszero>0$ such that every zero of $\Dfun$ with
$\sigma\le-\tszero$ is simple, and these zeros form two
complex-conjugate sequences $\{{\boldsymbol\rho}_{m},\overline{{\boldsymbol\rho}_{m}}\}_{m\ge m_{0}}$
satisfying
\be\label{eq:strings}
{\boldsymbol\rho}_{m}=1-\frac{i\pi(2m+1)}{\sL}+O(\er^{-\kappa m})
\ee
for some $\kappa>0$. In particular, the zeros recede along the two
rays $\arg(1-s)=\pm\Thstar$, consecutive zeros on each string have
limiting spacing $2\pi/|\sL|$, and
\[
\Bigl|\{\rho:\Dfun(\rho)=0,~\Re\,\rho<0,
~|\rho|\le R\}\Bigr|=\frac{|\sL|}{\pi}\,R+O(1).
\]
\end{theorem}

In the half-plane $\sigma<0$, Theorem~\ref{thm:fe} takes the form
\be\label{eq:reduction}
\Dfun(s)=\Gamma(w)\,\Zfun(w)
\qquad\text{with}\quad
\Zfun(w)\defeq\sum_{\om\in\Om}\om^{-w},
\ee
the series grouped by $k$ as in Remark~\ref{rem:order}. Since
$\Gamma(w)\ne0$, the zeros of $\Dfun$ with $\sigma<0$ are precisely
the zeros of $\Zfun$ with $u>1$. The coefficients of $\Dfun$ are
real, so the zero set is symmetric under conjugation and it suffices
to work in the quadrant $u>0$, $v\ge0$. Three lemmas carry the
argument. Lemma~\ref{lem:dominance} isolates $\omflat$ and $\omsharp$
as the only two frequencies in $\Om$ that can dominate $\Zfun(w)$ as
$u\to\infty$, Lemma~\ref{lem:g1} evaluates the $k=1$ column exactly,
and Lemma~\ref{lem:tail} reduces the remaining columns $k\ge2$ to the
single term $\omflat^{-w}$. Figure~\ref{fig:cones} displays the
resulting geometry.

\begin{figure}[ht]
\centering
\includegraphics[width=3.5truein]{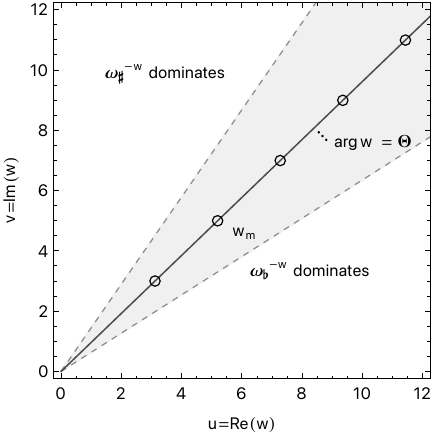}
\caption{The two dominance regions for $\Zfun(w)$ in the quadrant
$u>0$, $v\ge0$. Below the $\varepsilon$-cone about the ray
$\arg w=\Thstar$, the term $\omflat^{-w}$ dominates. Above the cone,
the term $\omsharp^{-w}$ dominates. The circles are the points
$w_{m}$ defined in the proof of Theorem~\ref{thm:strings}. They lie
on the ray, where neither term dominates. The zeros
${\boldsymbol\rho}_{m}$ of \eqref{eq:strings} correspond to these
points under $s=1-w$.}
\label{fig:cones}
\end{figure}

For every $\om\in\Om$, we have
$\bigl|\om^{-w}\bigr|=\exp\bigl\{-\Re(w\log\om)\bigr\}$, so the terms
of $\Zfun(w)$ that dominate as $u\to\infty$ correspond to the
minimizers of $\Re(w\log\om)$. The first lemma identifies these
minimizers in the two families $\{\ell=0\}$ and $\{\ell\ge1\}$.

\begin{lemma}\label{lem:dominance}
In the notation of \eqref{eq:freq}, we have
$\om_{k,\ell}=2\log k+2\pi i\ell$ for $(k,\ell)\ne(1,0)$.
For every $w=u+iv$ with $u>0$ and $v\ge0$,
\be\label{eq:dom1}
\Re\bigl(w\log\om_{k,0}\bigr)=u\log(2\log k)\ge u\log\omflat\qquad(k\ge2),
\ee
with equality if and only if $k=2$. For all $(k,\ell)$ with $\ell\ge 1$,
\be\label{eq:dom2}
\Re\bigl(w\log\om_{k,\ell}\bigr)-\Re\bigl(w\log\omsharp\bigr)
=u\log\frac{|\om_{k,\ell}|}{2\pi}
+v\Bigl\{\frac{\pi}{2}-\arg\om_{k,\ell}\Bigr\}
\ge u\log\ell\ge0,
\ee
where equality holds in the first inequality if and only if $k=1$,
and in the second if and only if $\ell=1$; in particular, the
left-hand side of \eqref{eq:dom2} vanishes if and only if
$\om_{k,\ell}=\omsharp$. Consequently, $\omflat$ and $\omsharp$ are the
unique minimizers of $\Re(w\log\om_{k,\ell})$ over $\{\ell=0\}$ and
$\{\ell\ge1\}$, respectively, for every such $w$.
\end{lemma}

\begin{proof}
The bound \eqref{eq:dom1} is immediate, since $\om_{k,0}=2\log k$ is
real and increasing with $k$.
For \eqref{eq:dom2}, the equality derives
from $\Re(w\log\om)=u\log|\om|-v\arg\om$; the first inequality
then follows because $|\om_{k,\ell}|\ge2\pi\ell$
and $\arg\om_{k,\ell}\in[0,\pi/2]$. The remaining assertions
follow directly, with equality throughout only when $k=1$.
\end{proof}

The next lemma computes the $k=1$ column of $\Zfun(w)$ exactly. The
identity restates Riemann's functional equation \eqref{eq:rfe} and
exhibits $\omsharp$ as the frequency contributed by $\zeta$, as
recorded in \eqref{eq:omsharp}.

\begin{lemma}\label{lem:g1}
For $u\ge2$ and $v\ge0$,
\[
\sum_{\ell\ne0}(2\pi i\ell)^{-w}
=\bigl\{\omsharp^{-w}+(-\omsharp)^{-w}\bigr\}\cdot\zeta(w)
=\omsharp^{-w}
+O\bigl(2^{-u}\bigl|\omsharp^{-w}\bigr|+(2\pi)^{-u}\bigr).
\]
\end{lemma}

\begin{proof}
By \eqref{eq:rfe} with $s=1-w$, we have
\[
\zeta(1-w)=\Gamma(w)\sum_{\ell\ne0}(2\pi i \ell)^{-w}
=\cX(1-w)\,\zeta(w).
\]
Since
\[
\cX(1-w)=2^{1-w}\pi^{-w}\Gamma(w)\cos(\pi w/2),
\]
the factor $\Gamma(w)$ cancels, and we get
\[
\sum_{\ell\ne0}(2\pi i \ell)^{-w}
=(2\pi)^{-w}\cdot 2\cos\Bigl(\frac{\pi w}{2}\Bigr)\cdot\zeta(w)
=\bigl\{\omsharp^{-w}+(-\omsharp)^{-w}\bigr\}\zeta(w),
\]
where we used the principal-branch identities
\[
(2\pi i)^{-w}=(2\pi)^{-w}\er^{-i\pi w/2}
\mand
(-2\pi i)^{-w}=(2\pi)^{-w}\er^{i\pi w/2}.
\]

The difference between the exact expression and $\omsharp^{-w}$ equals
\[
\omsharp^{-w}\bigl(\zeta(w)-1\bigr)+(-\omsharp)^{-w}\zeta(w).
\]
For $u\ge2$ and $v\ge0$, it is readily seen that
\[
|\zeta(w)-1|\le 2^{2-u}
\mand
\bigl|(-\omsharp)^{-w}\bigr|=(2\pi)^{-u}\er^{-\pi v/2}\le(2\pi)^{-u},
\]
and the first inequality implies $|\zeta(w)|\le1+2^{2-u}\le2$.
The stated bound follows.
\end{proof}

The last lemma controls the columns $k\ge2$ of $\Zfun(w)$, which
carry the contribution of the entire part $\Efun$, and extracts
$\omflat$ as their leading frequency, in agreement with
\eqref{eq:omflat}. The error is measured against the sum
$|\omflat^{-w}|+|\omsharp^{-w}|$ of the two main terms, since
either may dominate the other according to the position of $w$
relative to the ray $\arg w=\Thstar$.

\begin{lemma}\label{lem:tail}
There is an absolute constant $\eta>0$ such that, uniformly for
$w=u+iv$ with $u\ge2$ and $v\ge0$,
\be\label{eq:clippers}
\sum_{k\ge2}\,\sum_{\ell\in\Z}(2\log k+2\pi i\ell)^{-w}
=\omflat^{-w}
+O\Bigl\{\bigl(\bigl|\omflat^{-w}\bigr|
+\bigl|\omsharp^{-w}\bigr|\bigr)\er^{-\eta u}\Bigr\}.
\ee
\end{lemma}

\begin{proof}
We write $\Sigma(w)$ for the sum on the left side of \eqref{eq:clippers},
and we set
\[
W\defeq\bigl|\omflat^{-w}\bigr|+\bigl|\omsharp^{-w}\bigr|
=\omflat^{-u}+(2\pi)^{-u}\er^{\pi v/2}.
\]
By \eqref{eq:split}, \eqref{eq:rfe}, and \eqref{eq:fe2}, we have
$\Sigma(w)=\Efun(1-w)/\Gamma(w)$, where
\[
\Efun(1-w)=\sum_{m\in\N}\bigl(\zeta(2m)-1\bigr)m^{w-1}
\]
converges absolutely for every $w\in\C$. Writing
\[
\zeta(2m)-1=\sum_{j=2}^{30}j^{-2m}+r_{m},\qquad
r_{m}\defeq\sum_{j\ge31}j^{-2m},
\]
we obtain
\[
\Efun(1-w)=\sum_{j=2}^{30}\Li_{1-w}(j^{-2})+R(w),
\qquad
R(w)\defeq\sum_{m\in\N}r_{m}\,m^{w-1}.
\]

First, we bound $R$. Comparison with $\int_{31}^{\infty}t^{-2m}\dd t$
gives $r_{m}\le32\cdot31^{-2m}$, and therefore
\[
|R(w)|\le32\Li_{1-u}\bigl(31^{-2}\bigr).
\]
For $0<x<1$ and $\nu\defeq\log(1/x)$,
the function $t\mapsto t^{u-1}x^{t}$ is unimodal on $(0,\infty)$; hence
\[
\Li_{1-u}(x)\le\int_{0}^{\infty}t^{u-1}x^{t}\dd t
+\max_{t>0}\,t^{u-1}x^{t}
=\Gamma(u)\nu^{-u}+\Bigl(\frac{u-1}{\er\nu}\Bigr)^{u-1}
\le(1+\nu)\,\Gamma(u)\nu^{-u},
\]
where the last step uses $((u-1)/\er)^{u-1}\le\Gamma(u)$.
Taking $x=31^{-2}$, so $\nu=2\log31$, we get
$|R(w)|\ll\Gamma(u)(2\log31)^{-u}$. In the opposite direction,
Stirling's formula \cite[\S5.11]{DLMF} gives the uniform estimate
\[
\log|\Gamma(w)|
=(u-\tfrac12)\log|w|-v\arg w-u+\tfrac12\log2\pi+O(|w|^{-1})
\qquad(0\le\arg w\le\tfrac{\pi}{2},~|w|\ge2)
\]
and since $|w|\ge u$, it follows that
$\Gamma(u)/|\Gamma(w)|\ll\er^{\,v\arg w}$. Consequently,
\[
\frac{|R(w)|}{|\Gamma(w)|}
\ll(2\log31)^{-u}\,\er^{\,v\arg w}
\le(2\log31)^{-u}\,\er^{\pi v/2}
=\bigl|\omsharp^{-w}\bigr|
\exp\Bigl(-u\log\frac{\log31}{\pi}\Bigr).
\]

Next, fix any integer $j\in[2,30]$. Since $\Re(1-w)<0$, the Lipschitz formula
\eqref{eq:lipschitz} with $\lambda=2\log j$ shows
\[
\frac{\Li_{1-w}(j^{-2})}{\Gamma(w)}
=\sum_{\ell\in\Z}\om_{j,\ell}^{-w}
=(2\log j)^{-w}
+\sum_{\ell\le-1}\om_{j,\ell}^{-w}
+\sum_{\ell\ge1}\om_{j,\ell}^{-w}.
\]
For $\ell\le-1$, we have $\arg\om_{j,\ell}<0$ and
$|\om_{j,\ell}|\ge2\pi|\ell|$, therefore
\[
\sum_{\ell\le-1}\bigl|\om_{j,\ell}^{-w}\bigr|
\le(2\pi)^{-u}\zeta(u)\le2(2\pi)^{-u}
=2\,\omflat^{-u}\exp(-u\,\Re(\sL)),
\]
since $\log(2\pi/\omflat)=\Re(\sL)$.
For $\ell\ge1$, exponentiating the equality in \eqref{eq:dom2} and
discarding the nonnegative term
$v\bigl\{\tfrac{\pi}{2}-\arg\om_{j,\ell}\bigr\}$ yields
\[
\bigl|\om_{j,\ell}^{-w}\bigr|
\le\bigl|\omsharp^{-w}\bigr|
\exp\Bigl(-u\log\frac{|\om_{j,\ell}|}{2\pi}\Bigr).
\]
For $\ell=1$, the modulus $|\om_{j,1}|$ increases with $j$, and
$|\om_{2,1}|^{2}=\omflat^{2}+4\pi^{2}>(2\pi)^{2}$, so
\[
\log\frac{|\om_{j,1}|}{2\pi}
\ge\log\frac{|\om_{2,1}|}{2\pi}>0
\qquad(j\ge2).
\]
For $\ell\ge2$, the bound $|\om_{j,\ell}|\ge2\pi\ell$ gives
\[
\sum_{\ell\ge2}\bigl|\om_{j,\ell}^{-w}\bigr|
\le\bigl|\omsharp^{-w}\bigr|\sum_{\ell\ge2}\ell^{-u}
\le3\,\bigl|\omsharp^{-w}\bigr|\,2^{-u}.
\]
Hence
\[
\sum_{\ell\ge1}\bigl|\om_{j,\ell}^{-w}\bigr|
\le\bigl|\omsharp^{-w}\bigr|
\biggl(\exp\Bigl(-u\log\frac{|\om_{2,1}|}{2\pi}\Bigr)
+3\cdot 2^{-u}\biggr).
\]

Finally, for $3\le j\le30$, the main terms obtained above satisfy
\[
\bigl|(2\log j)^{-w}\bigr|=(2\log j)^{-u}\le(2\log3)^{-u}
=\omflat^{-u}\exp\Bigl(-u\log\frac{\log3}{\log2}\Bigr).
\]
Collecting the bound on $R$ and the twenty-nine columns, we conclude that
the estimate
\[
\frac{\Efun(1-w)}{\Gamma(w)}
=\omflat^{-w}+O\bigl(W\er^{-\eta u}\bigr)
\]
holds with the constant
\[
\eta\defeq\min\biggl\{\log\frac{\log31}{\pi},\,A,\,
\log\frac{|\om_{2,1}|}{2\pi},\,\log2,\,
\log\frac{\log3}{\log2}\biggr\}.
\]
As each entry of the minimum is positive, this proves the lemma.
\end{proof}

\begin{remark}\label{rem:eta}
Numerically, the minimum in the proof is its third entry, so the
argument establishes \eqref{eq:clippers} with
\[
\eta=\log\frac{|\om_{2,1}|}{2\pi}
=\frac12\log\Bigl(1+\frac{(\log 2)^2}{\pi^{2}}\Bigr)
=0.0237661\cdots,
\]
and with no larger exponent, since the estimate applied to the term
$\om_{2,1}^{-w}$ holds with equality at $v=0$.
This barrier, however, is an artifact of the proof. Retaining the
discarded factor
$\exp\bigl(-v\bigl\{\tfrac{\pi}{2}-\arg\om_{2,1}\bigr\}\bigr)$
raises the admissible exponent, in the sector about the ray
$\arg w=\Thstar$ relevant to Theorem~\ref{thm:strings}, toward
\[
\log\frac{|\om_{2,1}|}{2\pi}
+\frac{2\,\Re(\sL)}{\pi}\arctan\frac{\log2}{\pi}
=0.2326895\cdots,
\]
and no admissible exponent on the ray can exceed this value, since
the single term $\om_{2,1}^{-w}$ attains it there; we do not pursue
the refinement.
\end{remark}

\bigskip

\begin{proof}[Proof of Theorem~\ref{thm:strings}]
\vphantom{X}

\bigskip\noindent{\bf Step 1: Reduction.}
We recall that the zeros of $\Dfun$ with $\sigma<0$ are precisely the
zeros of $\Zfun$ with $u>1$, by \eqref{eq:reduction}; by
conjugation symmetry, it suffices to treat the case $v\ge0$.

\bigskip\noindent{\bf Step 2: Two-term reduction.}
By Lemmas~\ref{lem:g1} and~\ref{lem:tail},
\be\label{eq:twoterm}
\Zfun(w)=\omflat^{-w}+\omsharp^{-w}
+O\Bigl\{\bigl(\bigl|\omflat^{-w}\bigr|+\bigl|\omsharp^{-w}\bigr|\bigr)
\er^{-\eta u}\Bigr\}
\qquad(u\ge2,~v\ge0),
\ee
where $\eta>0$ is the constant of Lemma~\ref{lem:tail}. Indeed, the
error term of Lemma~\ref{lem:g1} equals
$\bigl|\omsharp^{-w}\bigr|2^{-u}+\bigl|\omflat^{-w}\bigr|\er^{-A\,u}$,
and, decreasing $\eta$ if necessary, we may assume
$\eta\le\min\{\log2,\,A\}$.
Since $\omsharp=\omflat\,\er^{\sL}$ by the definition of $\sL$, we may write
\be\label{eq:Zfunest1}
\Zfun(w)=\omflat^{-w}\Bigl\{1+\er^{-w\sL}
+O\Bigl(\bigl(1+\bigl|\er^{-w\sL}\bigr|\bigr)\er^{-\eta u}\Bigr)\Bigr\}.
\ee
Using the polar form $\sL=L\,\er^{i(\pi/2-\Thstar)}$ of
\eqref{eq:omom}, and writing $w=r\er^{i\theta}$, we have
\[
\Re\bigl(w\sL\bigr)=uA-\tfrac{\pi}{2}v=
L\,r\sin(\Thstar-\theta).
\]
Thus, for $w$ in the quadrant satisfying $|\theta-\Thstar|\ge\varepsilon$,
\[
\bigl|\er^{-w\sL}\bigr|=\er^{-uA+v\pi/2}
\begin{cases}
\le\,\er^{-L\,r\sin\varepsilon}
&\text{if $\theta\le\Thstar-\varepsilon$,}\\[2pt]
\ge\,\er^{L\,r\sin\varepsilon}
&\text{if $\theta\ge\Thstar+\varepsilon$.}
\end{cases}
\]
In either case,
once $L\,r\sin\varepsilon\ge\log2$, the modulus is bounded away
from one by a factor of at least two, and consequently
\[
\bigl|1+\er^{-w\sL}\bigr|\ge\tfrac13\bigl(1+\er^{-uA+v\pi/2}\bigr).
\]
Since $r=|w|\ge u$, this threshold and the domination of the error term
in \eqref{eq:Zfunest1} both hold for all $u$ large enough.
Hence, for every
$\varepsilon>0$ there is $U(\varepsilon)$ such that
$\Zfun(w)\ne0$ for $u\ge U(\varepsilon)$ outside the cone
$|\arg w-\Thstar|<\varepsilon$ (see Figure~\ref{fig:cones}). This proves
that $\Dfun$ is nonvanishing there and, together with conjugation,
confines every zero with $\sigma\le-U(\varepsilon)$ to the two conjugate
$\varepsilon$-cones about the rays $\arg w=\pm\Thstar$.
The confinement is mentioned only to orient the reader; the analysis
of Step~3 locates the zeros exactly and does not depend on it.

\bigskip\noindent{\bf Step 3: Location of the zeros.}
We denote
\[
z\defeq w\sL=x+iy,
\qquad
x\defeq\Re\,z=uA-\tfrac{\pi}{2}v,
\qquad
y\defeq\Im\,z=\tfrac{\pi}{2}u+Av.
\]
We define $g(z)\defeq1+\er^{-z}$; then \eqref{eq:Zfunest1} takes the
form
\be\label{eq:ZfunRouche}
\Zfun(w)\cdot\omflat^w=g(z)+h(z),
\ee
where the function $h$ so defined satisfies
\be\label{eq:ZfunRouche2}
h(z)\ll(1+\er^{-x})\er^{-\eta u}
\qquad(u\ge2,~v\ge0).
\ee
The zeros of $g$ are simple and occur at the points $z=i\pi(2k+1)$
with $k\in\Z$, or in $w$-coordinates, at
\[
w_{k}\defeq\frac{i\pi(2k+1)}{\sL}\qquad(k\in\Z).
\]
For $k=m\ge0$, the polar form \eqref{eq:omom} gives
$w_{m}=\pi(2m+1)L^{-1}\er^{i\Thstar}$, so these points are equally
spaced by $2\pi/L$ along the ray $\arg w=\Thstar$, with real parts
\[
u_{m}\defeq\Re\,w_{m}=\pi^{2}(2m+1)/(2L^{2}).
\]
As $z-i\pi(2k+1)=\sL(w-w_{k})$, the distance from $z$ to the nearest zero of $g$,
defined by
\[
d(z)\defeq\min_{k\in\Z}\bigl|z-i\pi(2k+1)\bigr|,
\]
satisfies
\[
d(z)=L\,\min_{k\in\Z}|w-w_{k}|.
\]

We claim that there is an absolute constant $c>0$ such that
\be\label{eq:offdisk}
\bigl|g(z)\bigr|\ge c\,\min\{1,\,d(z)\}\,\bigl(1+\er^{-x}\bigr)
\qquad(z\in\C).
\ee
For $|x|\ge\log2$, we have
\dalign{
|g(z)|&\ge1-\er^{-x}\ge\tfrac12\ge\tfrac13\bigl(1+\er^{-x}\bigr)
&&\qquad(x\ge\log2),\\
|g(z)|&\ge\er^{-x}-1\ge\tfrac12\,\er^{-x}\ge\tfrac13\bigl(1+\er^{-x}\bigr)
&&\qquad(x\le-\log2),
}
and \eqref{eq:offdisk} follows in this range since
$\min\{1,\,d(z)\}\le1$. For $|x|\le\log2$, we have $1+\er^{-x}\le3$,
and taking into account the $2\pi i$-periodicity of $g$,
we may assume $0\le y\le2\pi$,
in which case $d(z)=|z-i\pi|$. The quotient $g(z)/(z-i\pi)$ extends
to a zero-free analytic function on the compact rectangle
$|x|\le\log2$, $0\le y\le2\pi$, hence is bounded away from zero
there, and \eqref{eq:offdisk} again follows.

We set $r_{0}\defeq\pi/(2L)=0.7206\cdots$ and, for integers $m\ge0$, we denote
\[
w_{m}=u_{m}+iv_{m},\qquad
r_{m}\defeq\er^{-\eta u_{m}/2},\qquad
D_{m}\defeq\{w:|w-w_{m}|<r_{m}\}.
\]
We fix $m_{0}$ large enough that $r_{m}\le1/L$ (so $r_{m}<r_{0}$)
for all $m\ge m_{0}$. Since consecutive points
$w_{m}$ lie $2\pi/L=4r_0$ apart, the disks $\{D_{m}:m\ge m_{0}\}$
are pairwise disjoint. On the boundary $\partial D_{m}$, we have
$d(z)=L\,r_{m}\le1$, so \eqref{eq:offdisk} yields
\be\label{eq:gdisk}
\bigl|g(z)\bigr|
\ge c\,L\,\er^{-\eta u_{m}/2}\bigl(1+\er^{-x}\bigr)
\qquad(w\in\partial D_{m}).
\ee
On the other hand, the inequality $u\ge u_{m}-r_0>u_{m}-1$
shows that the error term in
\eqref{eq:ZfunRouche} satisfies
\be\label{eq:hdisk}
h(z)\ll\er^{\eta}\,\er^{-\eta u_{m}}\bigl(1+\er^{-x}\bigr)
\qquad(w\in\partial D_{m}).
\ee
The quotient of these two bounds is $\gg\er^{\eta u_{m}/2}$, which tends
to infinity with $m$; hence, enlarging $m_{0}$ if necessary, we
have $|h(z)|<|g(z)|$ at every point of $\partial D_{m}$ for every
$m\ge m_{0}$. The factor $\omflat^{-w}$
never vanishes, and the only zero of $g$ in the closure of $D_{m}$
is the simple zero $w_{m}$, since $r_{m}\le r_{0}$, and any
other zero lies at distance at least $4r_0$ from $w_{m}$.
Consequently, Rouch\'e's theorem, applied on $D_{m}$ to $g$ and $h$,
shows that $\Zfun$ has a unique zero $w'_{m}$ in $D_{m}$,
which is simple, with
\[
|w'_{m}-w_{m}|<r_{m}
=\exp\Bigl(-\frac{\eta\pi^{2}(2m+1)}{4L^{2}}\Bigr).
\]
We set ${\boldsymbol\rho}_{m}\defeq1-w'_{m}$, and \eqref{eq:strings}
follows, with any $\kappa\le\eta\pi^{2}/(2L^{2})$ admissible.

It remains to choose the constant $\tszero$ of the theorem and to
verify that no other zeros occur. The choice is made through finitely
many lower bounds on $\tszero$, each imposed where it arises.
Let $w=u+iv$, with $u\ge\tszero$ and $v\ge0$, lie outside the region
$\bigcup_{m\ge m_{0}}D_{m}$, and set $z\defeq w\sL$ as before.
We show that $|g(z)|>|h(z)|$, and thus $\Zfun(w)\ne0$ by
\eqref{eq:ZfunRouche}.

Since $\Im\,z>0$, every point $i\pi(2k+1)$ with $k\le-1$ lies at distance
greater than $\pi$ from~$z$; equivalently,
\[
|w-w_{k}|>\frac{\pi}{L}=2r_{0}\qquad (k\le-1).
\]
The argument now splits according to whether or not $w$ lies
within $r_{0}$ of some $w_{k}$.

\bigskip\noindent{\sc Case 1}: If $|w-w_{k}|\ge r_{0}$ for every
$k$, then $d(z)\ge L\,r_{0}=\pi/2$, so $\min\{1,d(z)\}=1$, and
\eqref{eq:offdisk} gives $|g(z)|\ge c\bigl(1+\er^{-x}\bigr)$. On the other
hand $h(z)\ll\bigl(1+\er^{-x}\bigr)\er^{-\eta\tszero}$ by
\eqref{eq:ZfunRouche2}. Thus, $|g(z)|>|h(z)|$ when $\tszero$ is large
enough.

\bigskip\noindent{\sc Case 2}: Suppose $|w-w_{m}|<r_{0}$ for some $m$;
note that $m\ge0$ since
$|w-w_{k}|>2r_{0}$ for $k\le-1$. Then $u_{m}\ge u-r_{0}\ge\tszero-1$,
so $m\ge m_{0}$ if $\tszero$ is sufficiently large.
Since $w\notin D_{m}$, we get $r_{m}\le|w-w_{m}|<r_{0}$, whence
$\min\{1,d(z)\}\ge L\,r_{m}$. The bounds \eqref{eq:gdisk} and
\eqref{eq:hdisk} therefore hold at $w$, the first by
\eqref{eq:offdisk} and the second because $u\ge u_{m}-1$, and we
conclude as before.

\bigskip
In both cases $|g(z)|>|h(z)|$, as required. In
combination with Step~1 and conjugation, this shows that every zero
of $\Dfun$ with $\sigma\le-\tszero$ is one of the
${\boldsymbol\rho}_{m}$ with $m\ge m_{0}$, or a conjugate, and that
all such zeros are simple. The argument is a self-contained instance
of the classical zero theory of two-term exponential sums surveyed by
Langer~\cite{Langer}.

\bigskip\noindent{\bf Step 4: Counting.}
We first verify that the strip $-\tszero\le\sigma<0$ contains only
finitely many zeros of $\Dfun$. For fixed $\delta\in(0,1)$, this
holds in the substrip $-\tszero\le\sigma\le-\delta$ by
Proposition~\ref{prop:middle}.
In the substrip $-\delta<\sigma<0$, every zero $\rho$ of $\Dfun$
satisfies $|\zeta(\rho)|=|\Efun(\rho)|\le\Efun(-1)$, the coefficients
of $\Efun$ being positive. On the other
hand, combining \eqref{eq:rfe} with the asymptotic
$|\cX(\sigma+it)|\asymp(|t|/2\pi)^{1/2-\sigma}$ and the classical
bound $1/\zeta(u+it)\ll(\log|t|)^{7}$ for $u\ge1$ and $|t|\ge2$
\cite[\S\S3.11, 4.12]{Titchmarsh}, we obtain
\[
|\zeta(s)|\gg|t|^{1/2}(\log|t|)^{-7}
\qquad(-\delta<\sigma<0,~|t|\ge2).
\]
Hence any zeros of $\Dfun$ in this substrip have bounded $|t|$
and thus are finite in number.

By Step~3, each point $w_{m}$ lies on the ray $\arg w=\Thstar$ at
distance $|w_{m}|=\pi(2m+1)/L$ from the origin, so the number of
$m\ge m_{0}$ with $|w_{m}|\le R$ is $RL/(2\pi)+O(1)$. Since
$|{\boldsymbol\rho}_{m}|=|w_{m}|+O(1)$ and consecutive $|w_{m}|$
differ by $2\pi/L$, the same count holds for the
${\boldsymbol\rho}_{m}$ with $|{\boldsymbol\rho}_{m}|\le R$. The
preceding paragraph and Step~3 together account for every zero of
$\Dfun$ with $\Re\,\rho<0$, so summing over the two conjugate
strings, we get
\[
\bigl|\{\rho:\Dfun(\rho)=0,~\Re\,\rho<0,~|\rho|\le R\}\bigr|=(L/\pi)R+O(1).
\]
This completes the proof.
\end{proof}

\begin{remark}
Conceptually, the left-plane zeros arise from interference between the
two summands of \eqref{eq:split}. The frequency $\omflat=2\log2$
leads the spectrum of $\Efun$, that is, of $\Li_{s}(1/4)$, while
$\pm\omsharp=\pm2\pi i$ lead that of $\zeta$, the $k=1$ column
$\cX(s)\zeta(1-s)$ of Theorem~\ref{thm:fe}. Along the negative real
axis, $\Efun$ dominates because $\omflat<|\omsharp|$, the mechanism
behind Proposition~\ref{prop:real}; as $v$ increases, the factor
$\er^{v\arg\om}$ favors $\zeta$, and the zeros settle on the
interface.
\end{remark}

\begin{table}[ht]
\centering
\setlength{\tabcolsep}{12pt}
\begin{tabular}{|c|l|l|l|}
\hline
\rule{0pt}{2.6ex}%
$m$ & computed zero ${\boldsymbol\rho}_{m}$ & asymptotic \eqref{eq:strings} & error \\[4pt]
\hline
\rule{0pt}{2.6ex}%
1 & $-2.77451-3.02888\,i$ & $-2.11589-2.99776\,i$ & $0.659$\\
2 & $-4.45043-5.15220\,i$ & $-4.19316-4.99627\,i$ & $0.301$\\
3 & $-6.35434-7.08692\,i$ & $-6.27042-6.99477\,i$ & $0.125$\\
4 & $-8.36660-9.04989\,i$ & $-8.34768-8.99328\,i$ & $0.060$\\
5 & $-10.42151-11.02708\,i$ & $-10.42494-10.99179\,i$ & $0.035$\\
6 & $-12.49329-13.01015\,i$ & $-12.50220-12.99030\,i$ & $0.022$\\
7 & $-14.57077-14.99819\,i$ & $-14.57947-14.98880\,i$ & $0.013$\\
8 & $-16.64997-16.99058\,i$ & $-16.65673-16.98731\,i$ & $0.008$\\
9 & $-18.72944-18.98606\,i$ & $-18.73399-18.98582\,i$ & $0.005$\\
10 & $-20.80859-20.98339\,i$ & $-20.81125-20.98432\,i$ & $0.003$\\[3pt]
\hline
\end{tabular}
\vspace{8pt}
\caption{The lower string of left-plane zeros of $\Dfun$ (conjugates omitted). Consecutive computed spacings increase from $2.705$ to $2.883$ and approach the limit $2\pi/L=2.882542\cdots$.}
\label{tab:left}
\end{table}

Table~\ref{tab:left} compares the zeros of $\Dfun$, computed from the
decomposition \eqref{eq:split}, with the leading term of
\eqref{eq:strings}; the error decays geometrically as predicted. The
imaginary parts in the table lie near the odd integers because the
coefficient $\Im(i\pi/\sL)$ in \eqref{eq:strings} equals
$\pi A/L^{2}=0.99925\cdots$, which is close to one.

{\Large\section{Zeros in the critical strip}\label{sec:strip}}

This section completes the description of the zero set of $\Dfun$
by treating the zeros in the critical strip. The key input is the
estimate \eqref{eq:Econst} below, which shows that on each
half-plane $\sigma\ge-\delta$, the entire part $\Efun$ is nearly the
constant $\zeta(2)-1$, so that the strip zeros of $\Dfun$ are
perturbed $a$-points of $\zeta$ for $a\approx-(\zeta(2)-1)$.
Theorem~\ref{thm:rvm} counts these zeros via the argument principle,
establishing a Riemann--von Mangoldt law, and Table~\ref{tab:strip}
pairs each strip zero of ordinate below $75$ with the $a$-point of
$\zeta$ that it shadows.

Throughout, we write
\[
\azero\defeq-(\zeta(2)-1)=1-\frac{\pi^{2}}{6}=-0.644934\cdots.
\]
For $\delta\ge0$, the approximation of $\Efun$ by the constant
$-\azero$ has the quantitative form
\be\label{eq:Econst}
\bigl|\Efun(s)+\azero\bigr|
\le\sum_{n\ge2}\bigl(\zeta(2n)-1\bigr)n^{\delta}
\eqdef r(\delta)
\qquad(\sigma\ge-\delta).
\ee
Numerically, we have
\[
r(\delta)=
\begin{cases}
0.105065\cdots&\text{if $\delta=0$,}\\
0.128496\cdots&\text{if $\delta=\tfrac14$,}\\
0.157663\cdots&\text{if $\delta=\tfrac12$,}\\
0.194149\cdots&\text{if $\delta=\tfrac34$,}\\
0.240032\cdots&\text{if $\delta=1$.}
\end{cases}
\]
Every zero $\rho$ of $\Dfun$ with $\Re\,\rho\ge-\delta$ satisfies
$\zeta(\rho)=-\Efun(\rho)$, an $a$-point of $\zeta$ for a value $a$
lying in the closed disk of radius $r(\delta)$ about $\azero$. The
value-distribution theory of $\zeta$ (Landau; see
\cite[Ch.~7]{Steuding}) then suggests, and the classical arguments
adapt to prove, the following result.

\bigskip

\begin{theorem}\label{thm:rvm}
Let $N_{\Dfun}(T)$ denote the number of zeros of $\Dfun$, counted
with multiplicity, that satisfy $-\tfrac14\le\sigma\le\szero$ and
$0<t\le T$. Then
\[
N_{\Dfun}(T)=\frac{T}{2\pi}\log\frac{T}{2\pi}-\frac{T}{2\pi}+O(\log T).
\]
\end{theorem}

\begin{proof}
We prove the asymptotic for those $T$ that are not ordinates of
zeros of $\Dfun$, removing this restriction at the end of the proof;
zeros are counted with multiplicity throughout.

We recall from \eqref{eq:righttri} that
\[
\bigl|\Dfun(s)-\zeta(2)\bigr|\le h(\sigma)
\qquad\text{with}\quad
h(\sigma)\defeq\sum_{n\ge2}\frac{\zeta(2n)}{n^{\sigma}}.
\]
We denote $\mu\defeq\zeta(2)-h(2)$. Since $h$ decreases on $(1,\infty)$ and
$h(\szero)=\zeta(2)$ with $\szero<2$, we see that $\mu$ is positive,
and therefore
\be\label{eq:rvmright}
\Re\,\Dfun(s)\ge\zeta(2)-h(\sigma)\ge\mu>0
\qquad(\sigma\ge2).
\ee

On the line $\sigma=-\tfrac14$, we combine the lower bound \eqref{eq:zetabound}
in the proof of Proposition~\ref{prop:middle}, applied with the choices
$C=\tfrac12$ and $\delta=\tfrac14$, along with the bound
$|\Efun(s)|\le\Efun(-\tfrac14)$, and we obtain
\be\label{eq:rvmleft}
\Dfun(s)=\zeta(s)\bigl(1+\Efun(s)/\zeta(s)\bigr)
\qquad\text{with}\quad
\Efun(s)/\zeta(s)\ll t^{-3/4}
\qquad(\sigma=-\tfrac14,~t\ge2).
\ee
We fix $t_{0}\ge2$ large enough that $|\Efun/\zeta|\le\tfrac12$ on
the ray $\sigma=-\tfrac14$, $t\ge t_{0}$. Increasing $t_{0}$ by
less than one, if necessary, we may assume that no zero of $\Dfun$ has
ordinate $t_{0}$.

Now let $T>t_{0}+8$ be free of ordinates of
zeros of $\Dfun$, and let $N_{\Dfun}^{\square}(T)$
denote the number of zeros of $\Dfun$ in the open rectangle
$\cR\defeq\{s\in\C:-\tfrac14<\sigma<2,~t_{0}<t<T\}$. We claim that
\be\label{eq:rvmreduce}
N_{\Dfun}(T)=N_{\Dfun}^{\square}(T)+O(1).
\ee
Indeed, since no zero has ordinate $T$, the two counts can differ
only through three types of zeros. Zeros with $\sigma\ge \szero$ do
not exist, by Proposition~\ref{prop:right}. Zeros on the line
$\sigma=-\tfrac14$ (which the open rectangle $\cR$ omits) have
$t\le t_{0}$ by \eqref{eq:rvmleft}. Zeros with $0<t\le t_{0}$ are
$O(1)$ in number, since the zeros of the meromorphic function
$\Dfun\not\equiv0$ do not accumulate.

\begin{figure}[ht]
\centering
\includegraphics[width=3.5truein]{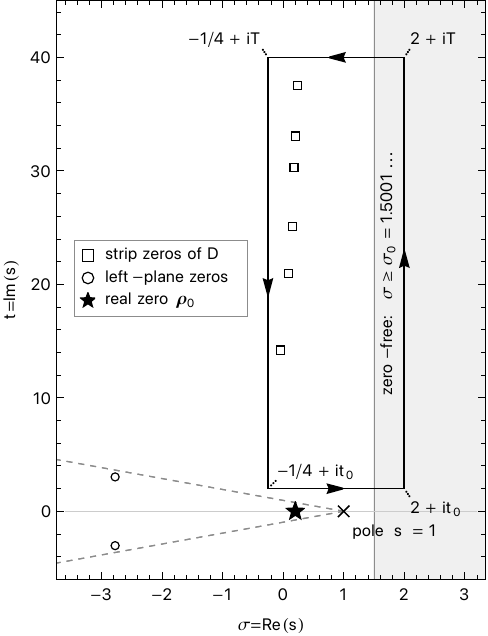}
\caption{The counting rectangle in the proof of
Theorem~\ref{thm:rvm}, drawn for $t_{0}=2$ and $T=40$. The squares
mark the strip zeros of $\Dfun$ with $t<40$, listed in
Table~\ref{tab:strip}. The cross marks the pole at $s=1$ and the
star marks the real zero $\rzero=0.2004\cdots$, both below the
bottom edge. The circles mark the conjugate pair
${\boldsymbol\rho}_{1}$, $\overline{{\boldsymbol\rho}_{1}}$ of
left-plane zeros, outside the rectangle. The dashed lines are the
rays $\arg(1-s)=\pm\Thstar$ of Theorem~\ref{thm:strings}. The shaded
half-plane $\sigma\ge\szero$ is zero-free.}
\label{fig:rvm}
\end{figure}

We now verify that the argument principle applies on $\cR$. The only
pole $s=1$ of $\Dfun$ lies below the bottom edge, so $\Dfun$ is
analytic on the closure of $\cR$. Moreover, no zero of $\Dfun$ lies
on the boundary $\partial\cR$ (see Figure~\ref{fig:rvm}); on the
right edge this holds by \eqref{eq:rvmright}, on the left and bottom
edges by \eqref{eq:rvmleft} and the choice of $t_{0}$, and on the
top edge by the choice of~$T$. Consequently,
\[
2\pi N_{\Dfun}^{\square}(T)=\Delta_{\partial}\arg\Dfun(s),
\]
where $\Delta_{\partial}$ denotes the increment of a continuous
branch of the argument along the positively oriented boundary. We
estimate the four edges separately.

\bigskip\noindent{\bf Step 1: Bottom and right.}
The bottom edge does not depend on $T$ and contributes $O(1)$. On
the right edge, the values of $\Dfun$ lie in the half-plane
$\Re\,z\ge\mu$ by \eqref{eq:rvmright}, so a continuous branch of
the argument remains in $(-\pi/2,\pi/2)$, and the increment is
$O(1)$.

\bigskip\noindent{\bf Step 2: Top.}
Let $n(T)$ denote the number of zeros of
the function $\sigma\mapsto\Re\,\Dfun(\sigma+iT)$ on
$[-\tfrac14,2]$. These zeros split the interval into at most
$n(T)+1$ subintervals, on each of which $\Re\,\Dfun(\sigma+iT)$
has constant sign. On each subinterval, the nonzero values of
$\Dfun(\sigma+iT)$ therefore lie in a closed half-plane, so the
increment of the argument across the top edge is at most
$\pi(n(T)+1)$ in absolute value.

To bound $n(T)$, we recall a standard consequence of Jensen's
formula (see \cite[\S9.4]{Titchmarsh}). If $f$ is analytic on the
disk $|z-z_{0}|\le R$ and $f(z_{0})\ne0$, then, for every
$r\in(0,R)$, the number of zeros of $f$ on the disk
$|z-z_{0}|\le r$, counted with multiplicity, does not exceed
\[
\frac{1}{\log(R/r)}\,
\log\frac{\max_{|z-z_{0}|\le R}|f(z)|}{|f(z_{0})|}.
\]
We set
\[
f(z)\defeq\tfrac12\bigl\{\Dfun(z+iT)+\Dfun(z-iT)\bigr\}.
\]
The poles of the two summands lie at $z=1\mp iT$, at distance
greater than four from the point $z=2$ (since $T>t_{0}+8$), so $f$ is analytic
on the disk $|z-2|\le 4$. The reflection identity
$\overline{\Dfun(\overline s)}=\Dfun(s)$, established in Step~1 of
the proof of Theorem~\ref{thm:strings}, shows that
\[
f(\sigma)=\Re\,\Dfun(\sigma+iT)
\]
for real $\sigma$; thus every zero counted by $n(T)$ is a zero of
$f$ on the segment $[-\tfrac14,2]$, which the disk
$|z-2|\le\tfrac94$ contains.
On the larger disk $|z-2|\le4$, we have
\[
-2\le\Re(z\pm iT)\le6
\mand
|\Im(z\pm iT)|\ge T-4,
\]
so the crude estimate $\zeta(w)\ll|\Im\,w|^{3}$, valid for $-2\le\Re\,w\le6$ and
$|\Im\,w|\ge2$ (see, for example, \cite[\S5.1]{Titchmarsh}),
together with $|\Efun(w)|\le\Efun(-2)$, yields
$f(z)\ll T^{3}$ on the disk. Finally, we have
$f(2)=\Re\,\Dfun(2+iT)\ge\mu$ by \eqref{eq:rvmright}. Applying the
above counting bound with $z_{0}=2$, $r=\tfrac94$, and $R=4$, we conclude that
\[
n(T)\le\frac{1}{\log(16/9)}\,
\log\frac{\max_{|z-2|\le4}|f(z)|}{\mu}\ll\log T,
\]
and so the top edge contributes $O(\log T)$.

\bigskip\noindent{\bf Step 3: Left.}
Now we treat the left edge, which is traversed downward from
$-\tfrac14+iT$ to $-\tfrac14+it_{0}$. Here and below, $\Delta$
denotes the increment as $t$ increases from $t_{0}$ to $T$, so the
edge contributes $-\Delta\arg\Dfun(-\tfrac14+it)$. By
\eqref{eq:rvmleft} and the choice of $t_{0}$, the factor
$1+\Efun/\zeta$ has real part at least $\tfrac12$ on the edge, so
its argument increments by $O(1)$ on the segment, hence \eqref{eq:rfe} implies
\[
\Delta\arg\Dfun(-\tfrac14+it)
=\Delta\arg\cX(-\tfrac14+it)
+\Delta\arg\zeta(\tfrac54-it)+O(1).
\]
For $\Re\,w=\tfrac54$, the series
\[
\log\zeta(w)=\sum_{p}\sum_{k\ge1}k^{-1}p^{-kw}
\]
converges absolutely, so $\arg\zeta(\tfrac54-it)$ admits a continuous
branch bounded by $\log\zeta(\tfrac54)$ in absolute value, whence
$\Delta\arg\zeta(\tfrac54-it)=O(1)$. For $\cX$,
the uniform asymptotic in a bounded strip
(see \cite[Eq.~(4.12.3)]{Titchmarsh}) gives
\[
\cX(\sigma+it)
=\Bigl(\frac{2\pi}{t}\Bigr)^{\sigma+it-1/2}\er^{i(t+\pi/4)}
\bigl(1+O(t^{-1})\bigr)
\qquad(-1\le\sigma\le2,~t\ge2).
\]
Taking arguments along a continuous branch at $\sigma=-\tfrac14$,
we find that
\[
\arg\cX(-\tfrac14+it)=-t\log\frac{t}{2\pi}+t+O(1)
\qquad(t\ge2).
\]
Since the branch value at $t=t_{0}$ is $O(1)$, the increment
satisfies
\be\label{eq:rvmstirling}
\Delta\arg\cX(-\tfrac14+it)=-T\log\frac{T}{2\pi}+T+O(1).
\ee
The left edge therefore contributes
\[
-\Delta\arg\Dfun(-\tfrac14+it)
=T\log\frac{T}{2\pi}-T+O(1).
\]

\bigskip

Combining the four edges with \eqref{eq:rvmreduce}, we conclude that
\[
N_{\Dfun}(T)
=\frac{T}{2\pi}\log\frac{T}{2\pi}-\frac{T}{2\pi}+O(\log T)
\]
whenever $T$ is not an ordinate of a zero. For arbitrary $T$, we
choose $T'\in(T,T+1)$ and $T''\in(T-1,T)$ that are not ordinates of
zeros. Since $N_{\Dfun}$ is nondecreasing, we have
$N_{\Dfun}(T'')\le N_{\Dfun}(T)\le N_{\Dfun}(T')$, and since the
main term varies by $O(\log T)$ on $[T-1,T+1]$, the asymptotic
holds for all $T$.
\end{proof}

\begin{remark}[Clustering near the critical line]
Levinson \cite{Levinson} proved that for fixed $a$ and every
$\varepsilon>0$, all but $o(N_{a}(T))$ of the $a$-points of $\zeta$ with
ordinate in $(0,T]$ lie in $|\sigma-\tfrac12|<\varepsilon$, where
$N_{a}(T)$ denotes their total number; the proof derives from
mean-value estimates for~$\zeta$ that are insensitive to replacing
the constant $a$ by the Dirichlet series $-\Efun(s)$, which is bounded,
nearly constant, and everywhere absolutely convergent.
Whether almost all strip zeros of $\Dfun$
cluster arbitrarily close to $\sigma=\tfrac12$ is left open here
(see Section~\ref{sec:concl}); the data below are consistent with
such clustering.
\end{remark}

Table~\ref{tab:strip} lists all strip zeros with $0<t<75$; there are
eighteen, against the prediction $18.54$ implied by
Theorem~\ref{thm:rvm} (with the customary $+7/8$).
Each zero lies within $0.08$ of the nearest
$a$-point of $\zeta$ for the fixed value $a=\azero$. The second
column verifies that $\zeta(\rho)$ lies within
$r(\tfrac14)=0.128\cdots$ of $\azero$, as \eqref{eq:Econst} demands.

\begin{table}[p]
\centering
\setlength{\tabcolsep}{6pt}
\begin{tabular}{|l|l|l|c|}
\hline
\multicolumn{1}{|c|}{\rule{0pt}{2.6ex}zero $\rho$ of $\Dfun$} & \multicolumn{1}{c|}{$\zeta(\rho)=-\Efun(\rho)$} & \multicolumn{1}{c|}{nearest $a$-point of $\zeta$} & distance \\[4pt]
\hline
\rule{0pt}{2.6ex}%
$-0.045481+14.202578\,i$ & $-0.5516-0.0302\,i$ & $-0.111989+14.239054\,i$ & $0.076$\\
$\phantom{-}0.091297+20.974154\,i$ & $-0.6041+0.0555\,i$ & $\phantom{-}0.075137+20.941496\,i$ & $0.036$\\
$\phantom{-}0.149094+25.089780\,i$ & $-0.6383-0.0644\,i$ & $\phantom{-}0.154382+25.117849\,i$ & $0.029$\\
$\phantom{-}0.179120+30.296803\,i$ & $-0.5996+0.0707\,i$ & $\phantom{-}0.173950+30.260653\,i$ & $0.037$\\
$\phantom{-}0.201071+33.068546\,i$ & $-0.6032-0.0678\,i$ & $\phantom{-}0.198225+33.101060\,i$ & $0.033$\\
$\phantom{-}0.233131+37.533769\,i$ & $-0.6758+0.0513\,i$ & $\phantom{-}0.246615+37.520026\,i$ & $0.019$\\
$\phantom{-}0.211171+40.868558\,i$ & $-0.5845+0.0074\,i$ & $\phantom{-}0.189184+40.862196\,i$ & $0.023$\\
$\phantom{-}0.258957+43.450622\,i$ & $-0.6506-0.0746\,i$ & $\phantom{-}0.272436+43.470819\,i$ & $0.024$\\
$\phantom{-}0.266535+47.832365\,i$ & $-0.6223+0.0768\,i$ & $\phantom{-}0.276408+47.806172\,i$ & $0.028$\\
$\phantom{-}0.231617+49.888034\,i$ & $-0.5756-0.0152\,i$ & $\phantom{-}0.209942+49.902407\,i$ & $0.026$\\
$\phantom{-}0.282691+53.008092\,i$ & $-0.6810-0.0455\,i$ & $\phantom{-}0.294379+53.017759\,i$ & $0.015$\\
$\phantom{-}0.287218+56.394482\,i$ & $-0.6621+0.0579\,i$ & $\phantom{-}0.295804+56.381027\,i$ & $0.016$\\
$\phantom{-}0.235011+59.236432\,i$ & $-0.5714-0.0030\,i$ & $\phantom{-}0.210079+59.228839\,i$ & $0.026$\\
$\phantom{-}0.306058+61.024145\,i$ & $-0.6284-0.0768\,i$ & $\phantom{-}0.319352+61.045757\,i$ & $0.025$\\
$\phantom{-}0.312599+64.990304\,i$ & $-0.6671+0.0691\,i$ & $\phantom{-}0.327629+64.978363\,i$ & $0.019$\\
$\phantom{-}0.258000+67.103361\,i$ & $-0.5885+0.0245\,i$ & $\phantom{-}0.240368+67.096845\,i$ & $0.019$\\
$\phantom{-}0.286127+69.580847\,i$ & $-0.6197-0.0479\,i$ & $\phantom{-}0.282000+69.594860\,i$ & $0.015$\\
$\phantom{-}0.320154+72.118275\,i$ & $-0.7004-0.0271\,i$ & $\phantom{-}0.333428+72.120615\,i$ & $0.013$\\[5pt]
\hline
\end{tabular}
\vspace{8pt}
\caption{Strip zeros of $\Dfun$ with \mbox{$0<t<75$} and the $a$-points of $\zeta$ that they shadow.}
\label{tab:strip}
\end{table}

{\Large\section{Computational details}\label{sec:comp}}

Every numerical claim in this paper, including the tables, the
plotted coordinates of the figures, and each standalone constant,
derives from one of three instrumented Mathematica runs. All three
were executed with Mathematica 13.0.0 on an eight-core ARM64
processor under macOS; each prints its working precision, truncation
points, stage wall-clock times, and certified counts, and the
printed outputs are retained together with the notebooks. The
default working precision is 40 significant digits, raised to 60 for
the string coordinates in the first run.

We evaluate $\Dfun$ through the decomposition \eqref{eq:split},
truncating the series for $\Efun$ at $n_{E}$ terms and
differentiating termwise where $\Dfun'$ is required. Since
$\zeta(2n)-1<4^{1-n}$, the omitted tail is negligible at working
precision provided $n_{E}$ is chosen against the leftmost abscissa
of the run; for instance, $n_{E}=220$ keeps the tail below
$10^{-62}$ on $\sigma\ge-30$. Completeness of the tables is
certified by counts obtained from the argument principle, with
$\Dfun'/\Dfun$ integrated numerically over the boundary of the
stated rectangle and the result rounded to the nearest integer,
matched in every case against an independent root search.

The first run consolidates the verification of the standalone
constants. It evaluates, to at least thirty printed digits, the
constants $A$, $L$, $\Thstar$, $2\pi/L$, and $\pi A/L^{2}$ of
Section~\ref{sec:left}, the exponent $\eta$ of
Lemma~\ref{lem:tail} together with the sector exponent of
Remark~\ref{rem:eta}, the admissible $\kappa$ of \eqref{eq:strings},
the abscissa $\szero$, the real zero $\rzero$ with the sign scan on
$(0,1)$, the table of $r(\delta)$ in \eqref{eq:Econst}, the margin
$\mu$ from the proof of Theorem~\ref{thm:rvm}, and the crossing
value $\alpha^{*}$ of Section~\ref{sec:concl}. It verifies the
numerical inputs $\er^{\pi}=23.140\cdots<31$ of
Lemma~\ref{lem:tail} and $0.177\cdots<1$ of
Proposition~\ref{prop:real}, checks two closed-form identities to
$10^{-39}$, and confirms the evaluations $\Efun(0)=\tfrac34$,
$\Efun(1)=\log2$, and $\Dfun(0)=\tfrac14$ with residuals below
$10^{-144}$ at truncation $240$.
Its final stage produces the string coordinates for $m\le25$ at 60
digits, from which every plotted point in the figures is drawn. The
full run completes in under twenty seconds.

The second run certifies Table~\ref{tab:left}. The
argument-principle count on the rectangle $[-22,-1]\times[-22,-1]$
equals ten, in agreement with the ten zeros subsequently located by
Newton iteration seeded at the leading term of \eqref{eq:strings};
the residuals $|\Dfun({\boldsymbol\rho}_{m})|$ vanish to the
attainable accuracy, which decreases from $10^{-38}$ at $m=1$ to
$10^{-21}$ at $m=10$ through cancellation in the series for
$\Efun$, and every published entry of the table is reproduced
within its final printed digit. Truncation $n_{E}=220$ suffices on
this rectangle. The certification consumes 48.5\,s of the 49.4\,s
total.

The third run certifies Table~\ref{tab:strip}. On the rectangle
$[-\tfrac14,2]\times[\tfrac12,75]$, whose right edge lies in the
half-plane of \eqref{eq:rvmright} and whose interior contains
neither the pole nor the real zero, the argument-principle count for
$\Dfun$ equals eighteen (393.2\,s), as does the count of $a$-points
of $\zeta$ for $a=\azero$ (327.6\,s).
A further count on $[-\tfrac14,2]\times[-\tfrac12,\tfrac12]$
returns zero (4.1\,s), the simple zero $\rzero$ canceling the
simple pole at $s=1$, so the band $0<t\le\tfrac12$ contains no
additional zeros.
A sweep locates the eighteen $a$-points with $t<75$, and twenty-one
with $t<80$, guarding the top edge, while a root search returns the
eighteen zeros, with residuals $|\Dfun(\rho)|$ and
$|\zeta(\rho)+\Efun(\rho)|$ below $10^{-34}$. Each zero lies within
$0.076$ of its nearest $a$-point, while the second nearest lies
beyond $1.79$, so the pairings of the table are unambiguous; every
published entry is reproduced within its final printed digit, and
the least ordinate gap to the edges of the rectangle exceeds
$0.54$. Truncation $n_{E}=130$ suffices for $\sigma\ge-\tfrac14$.
The run completes in 727.6\,s.

As an independent check, the zeros in Tables~\ref{tab:left}
and~\ref{tab:strip} and the constants $\rzero$ and $\szero$ were
recomputed with the arbitrary-precision library
mpmath~\cite{mpmath}, in full agreement at the stated precision.

\medskip

The notebooks and complete printed outputs are available from
the author on request.

{\Large\section{Concluding remarks}\label{sec:concl}}

Three questions seem natural. First, the dual series
$\Zfun(w)=\sum_{\om\in\Om}\om^{-w}$ of \eqref{eq:reduction} is a
Dirichlet series over the logarithms of the perfect squares; one can ask
whether there is a completed version of $\Dfun$ enjoying a genuine
symmetry, or whether the asymmetry between the harmonics $\{\log n\}$
and the frequencies $\{2\log k+2\pi i\Z\}$ is irreparable.

Second, a
full proof of Levinson-type clustering for the strip zeros of $\Dfun$,
and more finely their horizontal distribution relative to the
$a$-points they shadow, would complete the strip picture; the general
framework for zeros of $\zeta(s)-g(s)$ with $g$ a bounded Dirichlet
series has not, so far as we are aware, been worked out in the
literature.

Third, the same interference mechanism governs the wider family
\[
\Dfun_{\alpha}(s)\defeq\sum_{n\in\N}\zeta(\alpha n)n^{-s},
\]
defined for real $\alpha>1$, with frequencies $\alpha\log k+2\pi
i\ell$. For $\alpha<2\pi/\log2=9.0647\cdots$, the smallest
frequencies are $\alpha\log2$ and $\pm2\pi i$, and the mechanism
predicts left-plane strings along the rays of angle
$\arctan\bigl(2\log(2\pi/(\alpha\log2))/\pi\bigr)$; past the
threshold, the smallest pair is $\pm2\pi i$, and the string geometry
changes. Weighted variants change it differently; for
$\sum_{n\in\N}(\zeta(2n)-1)n^{-s}$, the $k=1$ column is absent, the
smallest frequencies are the real pair $2\log2$ and $2\log3$, and
the two-term prediction yields vertical strings in place of oblique
rays.

In the limit $\alpha\to\infty$, we have $\Dfun_{\alpha}\to\zeta$
locally uniformly; indeed $\Dfun_{\alpha}-\zeta\ll2^{-\alpha}$ on
compact sets, so the zeros of $\Dfun_{\alpha}$ migrate to those of
$\zeta$, the displacement at a simple zero $\rho$ being of size
$2^{-\alpha}/|\zeta'(\rho)|$. The
strip zeros also migrate as $a$-points, the target value
$a=1-\zeta(\alpha)$ tending to zero; the real zero leaves $(0,1)$
and settles beside the trivial zero at $-2$ once
\[
\Dfun_{\alpha}(0)=-\tfrac12+\sum_{k\ge2}\frac{1}{k^{\alpha}-1}
\]
changes sign, near $\alpha=2.3$. The strings flatten as $\alpha$
increases to $2\pi/\log2$,
past which the leading pair $\pm2\pi i$ regenerates the interference
behind the trivial zeros of $\zeta$ itself, and the string zeros
appear to be absorbed, through collisions of conjugate pairs on the
real axis, as the real zeros that settle beside the remaining
trivial zeros. A systematic treatment may be worthwhile.

\section*{Acknowledgments}
The author thanks Qiping Zhou for the proof of Proposition~\ref{prop:mono}.

Part of this work was carried out at the Max Planck Institute for
Mathematics in Bonn, whose hospitality and support the author
gratefully acknowledges.

The author used Claude (Anthropic) as a writing and verification
assistant in the preparation of this paper. Its contributions
comprised the drafting of Section~\ref{sec:comp} from the author's
certified computational output, revision of the prose for clarity
and consistency, and verification of cross-references, attributions,
and numerical constants against that output. All mathematical
results, proofs, and computations are the author's own, except as
attributed above, and the author reviewed and verified every part of
the text so prepared.


\begin{thebibliography}{99}

\bibitem{AL}
J.~A. Adell and A.~Lekuona,
Dirichlet's eta and beta functions: concavity and fast computation of
their derivatives.
\emph{J. Number Theory} 157 (2015), 215--222.

\bibitem{BBC}
J.~M. Borwein, D.~M. Bradley, and R.~E. Crandall,
Computational strategies for the Riemann zeta function.
\emph{J. Comput. Appl. Math.} 121 (2000), 247--296.

\bibitem{DH}
H.~Davenport and H.~Heilbronn,
On the zeros of certain Dirichlet series.
\emph{J. London Math. Soc.} 11 (1936), 181--185;
II, \emph{ibid.}, 307--312.

\bibitem{HTF}
A.~Erd\'elyi, W.~Magnus, F.~Oberhettinger, and F.~G. Tricomi,
\emph{Higher Transcendental Functions}, Vol.~I.
McGraw--Hill, New York, 1953.

\bibitem{mpmath}
F.~Johansson et al.,
\emph{mpmath: a Python library for arbitrary-precision floating-point arithmetic}, version 1.3.0, 2023.
\url{https://mpmath.org}

\bibitem{KR}
M.~Knopp and S.~Robins,
Easy proofs of Riemann's functional equation for $\zeta(s)$ and of Lipschitz summation.
\emph{Proc. Amer. Math. Soc.} 129 (2001), 1915--1922.

\bibitem{Langer}
R.~E. Langer,
On the zeros of exponential sums and integrals.
\emph{Bull. Amer. Math. Soc.} 37 (1931), 213--239.

\bibitem{Levinson}
N.~Levinson,
Almost all roots of $\zeta(s)=a$ are arbitrarily close to $\sigma=1/2$.
\emph{Proc. Nat. Acad. Sci. U.S.A.} 72 (1975), 1322--1324.

\bibitem{Lipschitz}
R.~Lipschitz,
Untersuchung der Eigenschaften einer Gattung von unendlichen Reihen.
\emph{J. Reine Angew. Math.} 105 (1889), 127--156.

\bibitem{vdL}
J.~van de Lune,
Some inequalities involving Riemann's zeta-function.
Report ZW 50/75, Mathematisch Centrum, Afdeling Zuivere Wiskunde,
Amsterdam, 1975.

\bibitem{DLMF}
F.~W.~J. Olver et al. (eds.),
\emph{NIST Digital Library of Mathematical Functions}.
\url{https://dlmf.nist.gov/}

\bibitem{SC}
H.~M. Srivastava and J.~Choi,
\emph{Series Associated with the Zeta and Related Functions}.
Kluwer, Dordrecht, 2001.

\bibitem{Steuding}
J.~Steuding,
\emph{Value-Distribution of $L$-Functions}.
Lecture Notes in Mathematics 1877, Springer, Berlin, 2007.

\bibitem{Titchmarsh}
E.~C. Titchmarsh,
\emph{The Theory of the Riemann Zeta-Function}, 2nd ed., revised by D.~R. Heath-Brown.
Oxford University Press, Oxford, 1986.

\end{thebibliography}
\end{document}